\documentclass{aic}

\aicAUTHORdetails{%
  title = {Rational Exponents Near Two}, 
  author = {David Conlon and Oliver Janzer},
  plaintextauthor = {David Conlon, Oliver Janzer},
  keywords = {extremal numbers, rational exponents},
}   

\aicEDITORdetails{%
   year={2022},
   number={9},
   received={12 March 2022},   
   published={23 December 2022},  
   doi={10.19086/aic.2022.9},      
}   

\usepackage{amsmath}
\usepackage{amssymb}
\usepackage{amsthm}
\usepackage{graphicx}
\usepackage{amscd}
\usepackage{epic, eepic}
\usepackage{url}
\usepackage{color}
\usepackage[utf8]{inputenc} 
\usepackage{comment}
\usepackage{enumerate}   
\usepackage{todonotes}
\usepackage{graphicx}
\usepackage{epstopdf}
\usepackage{enumitem}
\usepackage{tikz}

\makeatletter
\renewenvironment{proof}[1][\proofname] {\par\pushQED{\qed}\normalfont\topsep6\p@\@plus6\p@\relax\trivlist\item[\hskip\labelsep\bfseries#1\@addpunct{.}]\ignorespaces}{\popQED\endtrivlist\@endpefalse}
\makeatother

\usetikzlibrary{decorations.pathreplacing}
\tikzset{
  vertex/.style = {circle, draw, fill=white, inner sep=0pt, minimum width=4pt},
  root/.style = {circle, draw, fill=black, inner sep=0pt, minimum width=4pt},
  bracket/.style = {decorate,decoration={brace,amplitude=5pt},xshift=0pt,yshift=-10pt},
  bracket-r/.style = {decorate,decoration={brace,amplitude=5pt},xshift=10pt,yshift=0pt},
  bracket-u/.style = {decorate,decoration={brace,amplitude=5pt},xshift=0pt,yshift=10pt},
}

\newtheorem{theorem}{\bf Theorem}[section]
\newtheorem{lemma}[theorem]{\bf Lemma}

\newtheorem{conjecture}[theorem]{\bf Conjecture}

\theoremstyle{definition}

\newtheorem{definition}[theorem]{\bf Definition}

\def\eps{\varepsilon}

\def\cG{\mathcal{G}}
\def\cH{\mathcal{H}}

\def\ex{\mathrm{ex}}

\begin{document}

\begin{frontmatter}[classification=text]


\author[conlon]{David Conlon\thanks{Supported by NSF Award DMS-2054452.}}
\author[janzer]{Oliver Janzer\thanks{Supported by a research fellowship at Trinity College.}}

\begin{abstract}
A longstanding conjecture of Erd\H{o}s and Simonovits states that for every rational $r$ between $1$ and $2$ there is a graph $H$ such that the largest number of edges in an $H$-free graph on $n$ vertices is $\Theta(n^r)$. Answering a question raised by Jiang, Jiang and Ma, we show that the conjecture holds for all rationals of the form $2 - a/b$ with $b$ sufficiently large in terms of $a$.
\end{abstract}
\end{frontmatter}

\section{Introduction}

Given a positive integer $n$ and a graph $H$, the {\it extremal number} $\ex(n,H)$ is the largest number of edges in an $H$-free graph on $n$ vertices. In this short paper, we will be concerned with one of the standard conjectures about extremal numbers, the rational exponents conjecture of Erd\H{o}s and Simonovits (see, for example,~\cite{E81}), which states that every rational number $r$ between 1 and 2 is {\it realisable} in the sense that there exists a graph $H$ such that $\ex(n,H)=\Theta(n^r)$.

\begin{conjecture}[Rational exponents conjecture] \label{con:rational}
	For every rational number $r\in [1,2]$, there exists a graph $H$ with $\ex(n,H)=\Theta(n^r)$.
\end{conjecture}

The main result towards this conjecture is arguably the result of Bukh and Conlon~\cite{BC18} saying that for any rational number $r\in [1,2]$ there exists a finite family $\mathcal{H}$ of graphs such that $\ex(n,\mathcal{H})=\Theta(n^r)$, where $\ex(n,\mathcal{H})$ denotes the largest number of edges in an $n$-vertex graph which does not contain any $H\in \mathcal{H}$ as a subgraph. However, the conjecture remains open in its original form, which asks for a single graph rather than a family.

Nevertheless, following the breakthrough in~\cite{BC18}, progress on the single graph case has been swift, with substantial contributions, each extending the range of exponents for which the conjecture is known, made by Jiang, Ma and Yepremyan~\cite{JMY18}, Kang, Kim and Liu \cite{KKL18}, Conlon, Janzer and Lee~\cite{CJL21}, Janzer~\cite{Jan20}, Jiang and Qiu~\cite{JQ20,JQ19} and, most recently, Jiang, Jiang and Ma~\cite{JJM20}. For now, we highlight only one of these results, due to Jiang and Qiu~\cite{JQ19} saying that any rational of the form $1 + p/q$ with $q > p^2$ is realisable. 
Proving a conjecture of Jiang, Jiang and Ma~\cite[Conjecture 11]{JJM20} in a strong form, we show that a similar phenomenon holds near two.

\begin{theorem} \label{thm:main}
All rationals of the form $r = 2 - a/b$ with $b \geq \max(a, (a-1)^2)$ are realisable.
\end{theorem}

To say more, we must first explain the context in which the recent progress has been made. We will be interested in {\it rooted graphs} $(F, R)$ consisting of a graph $F$ together with a proper subset $R$ of the vertex set $V(F)$ that we refer to as the {\it roots}. We will usually just write $F$ if the roots are clear from context. For each $S \subseteq V(F) \setminus R$, let $\rho_F(S) := \frac{e_S}{|S|}$, where $e_S$ is the number of edges in $F$ incident with a vertex of $S$. 
The \emph{density} of $F$ is then $\rho(F) := \rho_F(V(F) \setminus R)$ and we say that $(F, R)$ is {\it balanced} if $\rho_F(S) \geq \rho(F)$ for all $S \subseteq V(F) \setminus R$. Finally, given a rooted graph $(F, R)$ and a positive integer $t$, the {\it $t$-blowup} $F^t$ is the graph obtained by taking $t$ vertex-disjoint copies of $F$ and identifying the different copies of $v$ for each $v \in R$. The following result of Bukh and Conlon~\cite{BC18} now yields a lower bound for the extremal number of $F^t$ provided $F$ is balanced and $t$ is sufficiently large in terms of $F$.

\begin{lemma}[Bukh--Conlon] \label{lem:BC}
For every balanced rooted graph $F$ with density $\rho$, there exists a positive integer $t_0$ such that $\ex(n, F^t)=\Omega(n^{2-\frac{1}{\rho}})$ for all integers $t \geq t_0$.
\end{lemma}

Paired to this result is the following conjecture, saying that Lemma~\ref{lem:BC} is tight up to the constant for balanced rooted {\it trees}. If true, this conjecture would easily imply Conjecture~\ref{con:rational}. 

\begin{conjecture}[Bukh--Conlon] \label{con:BC}
For every balanced rooted tree $F$ with density $\rho$ and all positive integers~$t$, $\ex(n, F^t) = O(n^{2-\frac{1}{\rho}})$.
\end{conjecture}

\begin{figure}
  \centering
  \begin{tikzpicture}[thick, scale=0.45, baseline=(v.base)]
    \coordinate (v) at (0,-0.8);
    \draw (-3.6,-1.1) -- (-4.2,-2.6) node[root]{};
    \draw (-3.6,-1.1) -- (-3.6,-2.6) node[root]{};
    \draw (-3.6,-1.1) -- (-3,-2.6) node[root]{};
    \draw [bracket] (-3,-2.6) -- (-4.2,-2.6) node[black,midway,yshift=-10pt]{\footnotesize $s$};
    \draw (-0.9,.4) -- (-3.6,-1.1) node[vertex]{};
    \draw (-1.8,-1.1) -- (-2.4,-2.6) node[root]{};
    \draw (-1.8,-1.1) -- (-1.8,-2.6) node[root]{};
    \draw (-1.8,-1.1) -- (-1.2,-2.6) node[root]{};
    \draw [bracket] (-1.2,-2.6) -- (-2.4,-2.6) node[black,midway,yshift=-10pt]{\footnotesize $s$};
    \draw (-0.9,.4) -- (-1.8,-1.1) node[vertex]{};
    \draw (0,-1.1) -- (-0.6,-2.6) node[root]{};
    \draw (0,-1.1) -- (0,-2.6) node[root]{};
    \draw (0,-1.1) -- (0.6,-2.6) node[root]{};
    \draw [bracket] (0.6,-2.6) -- (-0.6,-2.6) node[black,midway,yshift=-10pt]{\footnotesize $s$};
    \draw (-0.9,.4) -- (0,-1.1) node[vertex]{};
    \draw (1.8,-1.1) -- (1.2,-2.6) node[root]{};
    \draw (1.8,-1.1) -- (1.8,-2.6) node[root]{};
    \draw (1.8,-1.1) -- (2.4,-2.6) node[root]{};
    \draw [bracket] (2.4,-2.6) -- (1.2,-2.6) node[black,midway,yshift=-10pt]{\footnotesize $s$};
    \draw (-0.9,.4) node[vertex]{} -- (1.8,-1.1) node[vertex]{};
    \draw [bracket] (1.8,-3.8) -- (-3.6,-3.8) node[black,midway,yshift=-10pt]{\footnotesize $r$};    
  \end{tikzpicture}
  \caption{The rooted graph $F_{r,s}$, with black vertices representing roots.} \label{fig:F} 
\end{figure}
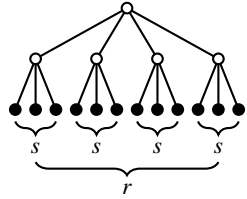

The recent progress then has centred on proving Conjecture~\ref{con:BC} for particular choices of the rooted tree $F$, with many novel and interesting ideas going into each new case. Here, we consider a family of rooted trees first studied in this setting by Jiang, Jiang and Ma~\cite{JJM20}. 
More precisely, for every pair of integers $(r,s)$ with $r, s \geq 1$, we write $F_{r, s}$ for the rooted graph with vertices $y$, $z_i$ for $1\leq i\leq r$ and $w_{i,j}$ for $1\leq i\leq r$, $1\leq j\leq s$, with the $w_{i,j}$ roots, and edges $y z_i$ for all $1\leq i\leq r$ and $z_i w_{i,j}$ for all $1\leq i\leq r,1\leq j\leq s$. For a picture with $r = 4$ and $s = 3$, we refer the reader to Figure~\ref{fig:F}, where the roots are drawn in black. It is easy to verify that $F_{r, s}$ is balanced provided $s \leq r$. Therefore, since $\rho(F_{r,s}) = (rs+r)/(r+1)$, Lemma~\ref{lem:BC} implies that 
$$\ex(n, F_{r,s}^t) = \Omega(n^{2- \frac{r+1}{rs+r}})$$ 
for $s \leq r$ and $t$ sufficiently large. Our main technical result is the corresponding upper bound for a certain range of parameters. 

\begin{theorem} \label{thm:main2}
	For any integers $r\geq s+2\geq 3$ and $t\geq 1$, $\ex(n,F_{r,s}^t)=O(n^{2-\frac{r+1}{rs+r}}).$
\end{theorem}

This improves on a result of Jiang, Jiang and Ma~\cite{JJM20}, who proved a result similar to Theorem~\ref{thm:main2}, but under the more restrictive assumption that $r \geq s^3 - 1$. While our argument, which we outline in the  next subsection, shares some ideas with theirs, it is considerably simpler.

To see that Theorem~\ref{thm:main2} implies Theorem~\ref{thm:main}, we require one more ingredient, a key observation of Kang, Kim and Liu~\cite{KKL18} saying that if the exponent $2 - \frac{a}{a p_0 +q}$ is realisable by a power of a balanced rooted graph, then so is $2 - \frac{a}{ap+q}$ for all $p \geq p_0$. But
\[2 - \frac{r+1}{rs+r} = 2 - \frac{r+1}{(r+1)s +(r- s)},\]
so the observation of Kang, Kim and Liu implies that the exponent $2 - \frac{r+1}{(r+1)p + (r- s)}$ is realisable for  all $p \geq s$. Since $r- s$ ranges from $2$ to $r-1$, this means that we get all exponents of the form $2 - \frac{r+1}{d}$ with $r \geq 3$, $d \geq r^2$ 
and $d \not\equiv -1, 0, 1 \pmod{r+1}$. 
Therefore, setting $a = r+1$, we see that $2 - \frac{a}{b}$ is realisable provided $a \geq 4$, $b \geq (a-1)^2$ and $b \not\equiv -1, 0, 1 \pmod{a}$. 
The remaining cases, where $a \in \{1, 2, 3\}$ or $b \equiv -1, 0, 1 \pmod{a}$, have all previously appeared in the literature (see, for instance,  \cite{KKL18}). 
It therefore remains to prove Theorem~\ref{thm:main2}.

\subsection{An outline of the proof}

Let $G$ be an $n$-vertex graph with $Cn^{2-\frac{r+1}{rs+r}}$ edges, where $C$ is taken sufficiently large in terms of $r$, $s$ and $t$. We want to show that $G$ contains $F_{r,s}^t$ as a subgraph. As is usual when estimating extremal numbers, we may assume that $G$ is $K$-almost-regular for some constant $K$ depending only on $r$ and $s$, by which we mean that every vertex in $G$ has degree at most $K$ times the minimum degree $\delta(G)$.

Suppose that $G$ does not contain $F_{r,s}^t$ as a subgraph. First, we show that, among all stars in $G$ with $s+1$ leaves, the proportion of those in which the leaves have codegree at least $|V(F_{r,s}^t)|$ is only $o(1)$. Indeed, otherwise we could find a vertex $u\in V(G)$ such that a positive proportion of the $(s+1)$-sets  in $N(u)$ have codegree at least $|V(F_{r,s}^t)|$. However, since $F_{r,s}^t$ is the subdivision of an $(s+1)$-partite $(s+1)$-uniform hypergraph, this would imply that $F_{r,s}^t$ can be embedded into $G$ with one part of the bipartition mapped to a subset of $N(u)$.

We call a copy of $F_{r,s}$ in $G$ \emph{nice} if, for each $1\leq i\leq r$, the codegree of the images of $y,w_{i,1},\dots,w_{i,s}$ is at most $|V(F_{r,s}^t)|$. By the previous paragraph and since $G$ is almost regular, almost all copies of $F_{r,s}$ in $G$ are nice.

Suppose now that we have a large collection of nice copies of $F_{r,s}$ in $G$ all of which have the same leaf set, i.e., they all map the $w_{i,j}$ to the same vertices $x_{i,j}$. Since $G$ is $F_{r,s}^t$-free, there cannot be $t$ of these copies of $F_{r,s}$ which are pairwise vertex-disjoint apart from the $x_{i,j}$. Hence, a positive proportion of them must map one of $y,z_1,\dots,z_r$ to the same vertex in $G$. However, there cannot exist many nice copies of $F_{r,s}$ which map $y$ and all the $w_{i,j}$ to the same set of vertices. Hence, we find that a positive proportion of the nice $F_{r,s}$ rooted at the $x_{i,j}$ must map some $z_k$ to the same vertex $v\in V(G)$. For the sake of notational simplicity, we will assume that a positive proportion of the copies rooted at the $x_{i,j}$ map $z_r$ to $v$. The crucial observation is that this means that $v$ sends many edges to a relatively small set that depends only on the vertices $x_{i,j}$ for $1\leq i\leq r-1,1\leq j\leq s$. More precisely, $v$ is clearly a neighbour of the image of $y$ in every copy of $F_{r,s}$ that maps $z_r$ to $v$. However, the locus of the possible images of $y$ is rather restricted: if $u$ is such an image, then, for each $1\leq i\leq r-1$, the vertices $u,x_{i,1},\dots,x_{i,s}$ have a common neighbour.

Fix a ``typical'' collection of vertices $x_{i,j}$, $1\leq i\leq r-1,1\leq j\leq s$, and let $X$ be the locus of the possible images of $y$ in embeddings of $F_{r,s}$ that map $w_{i,j}$ to $x_{i,j}$ for all $1\leq i\leq r-1,1\leq j\leq s$. For any $u\in X$, there are around $\delta(G)^{s+1}$ embeddings of $F_{r,s}$ that map $y$ to $u$ and $w_{i,j}$ to $x_{i,j}$ for each $1\leq i\leq r-1,1\leq j\leq s$, since we can ``freely" choose how $z_r,w_{r,1},\dots,w_{r,s}$ are embedded. If we assume that $|X|$ is about as large as it would be in a random graph with the same edge density, then, on average, for each embedding of $F_{r,s}$ which maps $w_{i,j}$ to $x_{i,j}$ for each $1\leq i\leq r-1,1\leq j\leq s$, there are a large constant number of copies of $F_{r,s}$ with the same leaves. Assuming that these copies are nice, the previous paragraph shows that there are many embeddings of $F_{r,s}$ which map $w_{i,j}$ to $x_{i,j}$ for all $1\leq i\leq r-1,1\leq j\leq s$ with the property that the image of $z_r$ has a large constant number of neighbours in $X$. This then allows us to conclude that there are many edges $uv\in E(G)$ with $u\in X$ such that $v$ has a large constant number of neighbours in $X$, which in turn yields a very unbalanced bipartite subgraph of $G$ with parts $X$ and $Y$ where every $v\in Y$ has many neighbours in $X$. This subgraph contains many stars with $s+1$ leaves centred in $Y$ and, for most of them, the leaves have large codegree, contradicting the observation made in the second paragraph.

\section{Proof of Theorem \ref{thm:main2}}

Fix $r\geq s+2\geq 3$, $t\geq 1$ and let $H=F_{r,s}^t$. We begin  our proof by defining what it means for a star with $s+1$ leaves to be heavy and then showing that there cannot be too many such stars. Originating in work of Conlon and Lee~\cite{CL21} and Janzer~\cite{Jan19} on extremal numbers of subdivisions, similar definitions and results appear often in the recent literature on the rational exponents conjecture.

\begin{definition}
	We call a star with $s+1$ leaves \emph{heavy} if the leaves have codegree at least $|V(H)|$ and \emph{light} otherwise.
\end{definition}

\begin{lemma} \label{lem:heavy stars}
	For any $\eps>0$, there is a constant $C=C(\eps,H)$ such that the following holds. Let $G$ be an $H$-free bipartite graph with parts $X$ and $Y$ and minimum degree at least $C$ on side $Y$. Then the proportion of heavy $(s+1)$-stars among all $(s+1)$-stars centred in $Y$ is at most $\eps$.
\end{lemma}

\begin{proof}
	It suffices to prove that for each $u\in Y$, the proportion of heavy stars among all stars centred at $u$ is at most $\eps$. Define an $(s+1)$-uniform hypergraph $\cG$ on vertex set $N(u)$ by setting $S\subset N(u)$ with $|S| = s+1$ to be an edge of $\cG$ if and only if the common neighbourhood (in $G$) of the vertices in $S$ has order at least $|V(H)|$. We also define an $(s+1)$-uniform hypergraph $\cH$ with vertices $y_k$ for $1\leq k\leq t$ and $w_{i,j}$ for $1\leq i\leq r,1\leq j\leq s$ whose edges are $\{y_kw_{i,j}: 1\leq j\leq s\}$ for every $1\leq k\leq t,1\leq i\leq r$. It is easy to see that if $\cG$ contains a copy of $\cH$, then there exists a copy of $H$ in $G$. Moreover, $\cH$ is $(s+1)$-partite (the parts being $\{y_1,\dots,y_t\}$ and $\{w_{i,j}:1\leq i\leq r\}$ for each $1\leq j\leq s$), so $\ex(n,\cH)=o(n^{s+1})$. It follows that if $|N(u)|$ is large enough in terms of $\eps$ and $\cH$, then there are at most $\eps \binom{|N(u)|}{s+1}$ heavy $(s+1)$-stars in $G$ with centre $u$. Since $\cH$ depends only on $H$, the proof is complete.
\end{proof}

We now make a few definitions which capture some of the main ideas in our proof.

\begin{definition}
	Let $F$ be a labelled copy of $F_{r,s}$ with vertices $y,z_i,w_{i,j}$ as before. We call $F$ \emph{nice} if, for each $1\leq i\leq r$, the $(s+1)$-star with centre $z_i$ and leaves $y,w_{i,1},\dots,w_{i,s}$ is light.
\end{definition}

\begin{definition} \label{def:locus}
	\sloppy For distinct vertices $x_{i,j}$ with $1\leq i\leq r-1$, $1\leq j\leq s$ in a graph $G$, let $S(x_{1,1},\dots,x_{1,s},x_{2,1},\dots,x_{2,s},\dots,x_{r-1,1},\dots,x_{r-1,s})$ be the set of vertices $u\in V(G)$ for which there are vertices $v_1,\dots,v_{r-1}$ such that $u$, the $v_i$ and the $x_{i,j}$ are all distinct, $uv_i\in E(G)$ for all $i$ and $v_ix_{i,j}\in E(G)$ for all $i,j$.
\end{definition}

\begin{definition}
	\sloppy Let $F$ be a nice labelled copy of $F_{r,s}$ with vertices $y,z_i,w_{i,j}$ and let $q$ be the number of nice labelled copies of $F_{r,s}$ with the same labelled leaf set as $F$. For $c>0$ and $1\leq k\leq r$, we call $F$ \emph{$(c,k)$-rich} if $z_k$ has at least $cq$ neighbours in $S(w_{1,1},\dots,w_{1,s},\dots,w_{k-1,1},\dots,w_{k-1,s},w_{k+1,1},\dots,w_{k+1,s},\dots,w_{r,1},\dots,w_{r,s})$.
\end{definition}

The next lemma shows that if an $H$-free graph $G$ has many nice copies of $F_{r,s}$ sharing the same leaves, then many of those copies of $F_{r,s}$ are rich.

\begin{lemma} \label{lem:rich with fixed leaves}
	There exist positive constants $c=c(H)$ and $C=C(H)$ such that the following holds. Let $G$ be an $H$-free graph and let $x_{i,j}$, for $1\leq i\leq r$, $1\leq j\leq s$, be distinct vertices in $G$. Assume that there are $q\geq C$ nice labelled copies of $F_{r,s}$ in $G$ with $w_{i,j}$ mapped to $x_{i,j}$ for all $i,j$. Then there is some $1\leq k\leq r$ such that the number of $(c,k)$-rich labelled copies of $F_{r,s}$ with $w_{i,j}$ mapped to $x_{i,j}$ for all $i, j$ is at least $cq$.
\end{lemma}

\begin{proof}
    Let $C=(t-1)(r+1)^2|V(H)|^r+1$ and $c=1/((t-1)(r+1)^2|V(H)|^r)$. Since $G$ is $H$-free, there cannot be more than $t-1$ copies of $F_{r,s}$ which all have the same leaves $x_{i,j}$ but are otherwise pairwise vertex-disjoint. This means that any maximal collection of copies of $F_{r,s}$ with leaves $x_{i,j}$ which are otherwise pairwise disjoint cover a set $R$ of at most $(t-1)(r+1)$ vertices in addition to $\{x_{i,j}:1\leq i\leq r,1\leq j\leq s\}$. Because of the maximality, any labelled copy of $F_{r,s}$ with leaves $x_{i,j}$ must map one of $y,z_1,\dots,z_r$ to an element of $R$. By the pigeonhole principle, there are therefore at least $q/(|R|(r+1))\geq q/((t-1)(r+1)^2)$ nice copies of $F_{r,s}$ with leaves $x_{i,j}$ in which one of the vertices $y,z_1,\dots,z_r$ is mapped to the same vertex $v$ in $G$. By the condition that these copies are nice, $y$ cannot be mapped to the same vertex in more than $|V(H)|^r$ copies. Hence, since $q\geq C>(t-1)(r+1)^2|V(H)|^r$, 
    there is some $1\leq k\leq r$ such that $z_k$ is mapped to the same vertex $v$ in at least $q/((t-1)(r+1)^2)$ copies. Again using the fact that $y$ is mapped to the same vertex at most $|V(H)|^r$ many times, it follows that there are at least $q/((t-1)(r+1)^2|V(H)|^r)=cq$ different images of $y$ in these copies. All of these vertices are in $S(x_{1,1},\dots,x_{1,s},\dots,x_{k-1,1},\dots,x_{k-1,s},x_{k+1,1},\dots,x_{k+1,s},\dots,x_{r,1},\dots,x_{r,s})$ and all of them are neighbours of $v$. Thus, all nice copies of $F_{r,s}$ mapping $w_{i,j}$ to $x_{i,j}$ for every $i,j$ and $z_k$ to $v$ are $(c,k)$-rich.
\end{proof}

The upshot of what we have done so far is the following lemma, which says that, under a mild technical condition on the degrees (that we will in any case be able to assume), any $H$-free graph must have many rich copies of $F_{r,s}$.

\begin{lemma} \label{lem:many rich}
	For any positive real number $K$, there are positive constants $c=c(H)$ and $C=C(K,H)$ such that the following holds. Let $G$ be an $H$-free $n$-vertex bipartite graph with minimum degree $\delta\geq Cn^{1-\frac{r+1}{rs+r}}$ and maximum degree at most $K\delta$. Then $G$ has at least $cn\delta^{rs+r}$ $(c,r)$-rich labelled copies of $F_{r,s}$.
\end{lemma}

\begin{proof}
	The number of labelled copies of $F_{r,s}$ in $G$ is at least $\frac{1}{2}n\delta^{rs+r}$. Let $\eps=\frac{1}{4rK^{rs+r}}$. By Lemma \ref{lem:heavy stars}, if $C$ is sufficiently large compared to $K$ and $H$, then the proportion of heavy $(s+1)$-stars in $G$ is at most $\eps$. Then, by the maximum degree condition, there are at most $\eps n(K\delta)^{s+1}$ labelled heavy $(s+1)$-stars. Thus, again using the maximum degree assumption, there are at most $r\cdot \eps n (K\delta)^{s+1}\cdot (K\delta)^{rs+r-(s+1)}=\frac{1}{4}n\delta^{rs+r}$ labelled copies of $F_{r,s}$ in $G$ which contain a heavy $(s+1)$-star. It follows that there are at least $\frac{1}{4}n\delta^{rs+r}\geq \frac{C^{rs+r}}{4}n^{rs}$ nice labelled copies of $F_{r,s}$ in $G$. Let $C'$ be the constant $C(H)$ from Lemma \ref{lem:rich with fixed leaves}. Clearly, there are at most $C'n^{rs}$ nice labelled copies of $F_{r,s}$ whose leaves $w_{i,j}$ are mapped to some $x_{i,j}$ for all $1\leq i\leq r$, $1\leq j\leq s$ with the property that there are fewer than $C'$ nice labelled copies of $F_{r,s}$ with $w_{i,j}$ mapped to $x_{i,j}$. Hence, if $C$ is sufficiently large, then these nice labelled copies of $F_{r,s}$ amount to at most half of all nice labelled copies of $F_{r,s}$. The statement then follows from Lemma \ref{lem:rich with fixed leaves} by noting that the number of $(c,k)$-rich labelled copies of $F_{r,s}$ in $G$ is the same for every $k$.
\end{proof}

The following lemma is the last ingredient needed for the proof of Theorem \ref{thm:main2}.

\begin{lemma} \label{lem:dependent random choice}
	There is a constant $C_0=C_0(H)$ such that the following holds. Let $G$ be a bipartite graph with parts $X$ and $Y$ such that there are at least $|X|p$ edges $xy$ for which $x\in X$, $y\in Y$ and $y$ has degree at least $q$ in $G$. If $q\geq C_0$ and $pq^s\geq C_0|X|^s$, then $G$ contains $H$ as a subgraph.
\end{lemma}

We will prove Lemma \ref{lem:dependent random choice} using Lemma \ref{lem:heavy stars}, but we remark that it can also be proved directly using dependent random choice.

\begin{proof}
    We may assume, by shrinking $Y$ if necessary, that each $y\in Y$ has degree at least $q$. Then any edge in $G$ can be extended in at least $\binom{q-1}{s}$ ways to an $(s+1)$-star centred in $Y$. Hence, the conditions of the lemma guarantee that $G$ has at least $|X|p\binom{q-1}{s}/(s+1)$ stars with $s+1$ leaves centred in $Y$. Suppose that $G$ is $H$-free. If $C_0$ is sufficiently large, then Lemma \ref{lem:heavy stars} implies that at least half of the $(s+1)$-stars centred in $Y$ are light. If again $C_0$ is sufficiently large, then, since $pq^s\geq C_0|X|^s$, there are more than $|V(H)||X|^{s+1}$ light $(s+1)$-stars centred in $Y$. However, since there are at most $|X|^{s+1}$ choices for the set of $s+1$ leaves and, given such a choice, there are at most $|V(H)|$ possibilities for the centre, this is a contradiction.
\end{proof}

We are now ready to complete the proof of Theorem~\ref{thm:main2}. By a reduction going back to work of Erd\H{o}s and Simonovits~\cite{ES70}, we may assume that our graph is {\it $K$-almost-regular} for some constant $K$ depending only on $r$ and $s$, by which we mean that $\max_{v \in V(G)} \deg(v) \leq K \min_{v \in V(G)} \deg(v)$. As noted in~\cite{CL21}, we may also assume that the graph is bipartite, reducing our task to proving the following result.

\begin{theorem} \label{thm:regular}
	For any positive real number $K$, there is a constant $C=C(K,H)$ such that if $G$ is an $n$-vertex bipartite graph with minimum degree $\delta \geq Cn^{1-\frac{r+1}{rs+r}}$ and maximum degree at most $K\delta$, then $G$ contains $H$ as a subgraph.
\end{theorem}

\begin{proof}
	Let $C$ be sufficiently large and suppose, for the sake of contradiction, that $G$ is $H$-free. By Lemma~\ref{lem:many rich}, there is a positive constant $c=c(H)$ such that $G$ has at least $cn\delta^{rs+r}$ $(c,r)$-rich labelled copies of $F_{r,s}$.
	
	\medskip
	
	\noindent \emph{Claim.} There are distinct vertices $x_{i,j}\in V(G)$ for $1\leq i\leq r-1$, $1\leq j\leq s$ such that the number of $(c,r)$-rich labelled copies of $F_{r,s}$ mapping $w_{i,j}$ to $x_{i,j}$ for $1\leq i\leq r-1,1\leq j\leq s$ is
	\begin{enumerate}
		\item at least $\frac{1}{2}cn\delta^{rs+r}n^{-(r-1)s}$ and \label{property:many ext}
		\item at least $c/(2K^{rs+r})$ times the number of all labelled copies of $F_{r,s}$ mapping $w_{i,j}$ to $x_{i,j}$ for $1\leq i\leq r-1,1\leq j\leq s$. \label{property:relatively many ext}
	\end{enumerate}
	
	\medskip
	
	\noindent \emph{Proof of Claim.} Clearly, the number of $(c,r)$-rich labelled copies of $F_{r,s}$ which agree with fewer than $\frac{1}{2}cn\delta^{rs+r}n^{-(r-1)s}$ $(c,r)$-rich labelled copies of $F_{r,s}$ on the images of $w_{i,j}$ ($1\leq i\leq r-1,1\leq j\leq s$) is less than $\frac{1}{2}cn\delta^{rs+r}$. Hence, there are at least $\frac{1}{2}cn\delta^{rs+r}$ $(c,r)$-rich labelled copies of $F_{r,s}$ such that each of them agrees with at least $\frac{1}{2}cn\delta^{rs+r}n^{-(r-1)s}$ other $(c,r)$-rich labelled copies of $F_{r,s}$ on the images $w_{i,j}$ ($1\leq i\leq r-1,1\leq j\leq s$). Moreover, the total number of labelled copies of $F_{r,s}$ in $G$ is at most $n(K\delta)^{rs+r}$. Since $\frac{\frac{1}{2}cn\delta^{rs+r}}{n(K\delta)^{rs+r}}=c/(2K^{rs+r})$, there are vertices $x_{i,j}$ satisfying the two conditions in the claim. \hfill $\Box$
	
	\medskip
	
	Fix some vertices $x_{i,j}$ ($1\leq i\leq r-1,1\leq j\leq s$) satisfying the conclusion of the claim and let $X=S(x_{1,1},\dots,x_{1,s},x_{2,1},\dots,x_{2,s},\dots,x_{r-1,1},\dots,x_{r-1,s})$. Moreover, let $\mathcal{A}$ be the set of $(c,r)$-rich labelled copies of $F_{r,s}$ mapping $w_{i,j}$ to $x_{i,j}$ for all $1\leq i\leq r-1,1\leq j\leq s$.
	Observe that
	\begin{equation}
	    |\mathcal{A}|\leq |X|(K\delta)^{s+1}|V(H)|^{r-1}. \label{eqn:upper bound on N}
	\end{equation}
	Indeed, there are at most $|X|$ ways to embed $y\in V(F_{r,s})$, by the maximum degree condition there are at most $(K\delta)^{s+1}$ ways to embed $z_r,w_{r,1},w_{r,2},\dots,w_{r,s}$ and, finally, since the copy needs to be nice, there are at most $|V(H)|$ ways to embed each of $z_1,z_2,\dots,z_{r-1}$. On the other hand, property \ref{property:many ext} of the claim asserts that $|\mathcal{A}|\geq \frac{1}{2}cn\delta^{rs+r}n^{-(r-1)s}$, so, by comparing this with (\ref{eqn:upper bound on N}), we get
	\begin{equation}
	    |X|(K\delta)^{s+1}|V(H)|^{r-1}\geq \frac{1}{2}cn\delta^{rs+r}n^{-(r-1)s}. \label{eqn:lower bound on |X|}
	\end{equation}
	
	Note also that the total number of labelled copies of $F_{r,s}$ mapping $w_{i,j}$ to $x_{i,j}$ for all $1\leq i\leq r-1,1\leq j\leq s$ is at least $|X|\delta^{s+1}/2$, since, after embedding $y$ to any vertex in $X$, there are at least $\delta^{s+1}/2$ ways to complete the embedding. It follows from property \ref{property:relatively many ext} of the claim that
	\begin{equation*}
	    |\mathcal{A}|\geq \frac{c}{4K^{rs+r}}|X|\delta^{s+1}.
	\end{equation*}
	The number of those elements of $\mathcal{A}$ which agree with fewer than $\frac{c}{8K^{rs+r}}|X|\delta^{s+1}n^{-s}$ elements of $\mathcal{A}$ on the images of $w_{r,1},\dots,w_{r,s}$ is at most $\frac{c}{8K^{rs+r}}|X|\delta^{s+1}$. Hence, there are at least $\frac{c}{8K^{rs+r}}|X|\delta^{s+1}$ elements of $\mathcal{A}$ such that each of them agrees with at least $\frac{c}{8K^{rs+r}}|X|\delta^{s+1}n^{-s}$ elements of $\mathcal{A}$ on the images of $w_{r,1},\dots,w_{r,s}$. By the definition of $(c,r)$-richness, for all these copies, the image of $z_r$ has at least $c\cdot \frac{c}{8K^{rs+r}}|X|\delta^{s+1}n^{-s}$ neighbours in $X$. By the maximum degree condition in $G$ and since any $(c,r)$-rich copy of $F_{r,s}$ is nice, we see that for any $u,v\in V(G)$, there are at most $|V(H)|^{r-1}(K\delta)^s$ elements of $\mathcal{A}$ which map $y$ to $u$ and $z_r$ to $v$.
Hence, $G$ has at least $\frac{\frac{c}{8K^{rs+r}}|X|\delta^{s+1}}{|V(H)|^{r-1}(K\delta)^s}=\frac{c}{8K^{rs+r+s}|V(H)|^{r-1}}|X|\delta$ edges $uv$ with $u\in X$ and $v\in V(G)$ such that $v$ has at least $c\cdot \frac{c}{8K^{rs+r}}|X|\delta^{s+1}n^{-s}$ neighbours in $X$. Set $Y=V(G)\setminus X$. Since $G$ is bipartite, any neighbour of a vertex in $X$ is in $Y$.
	
	We now want to apply Lemma \ref{lem:dependent random choice} to the bipartite graph $G\lbrack X,Y\rbrack$. By the previous paragraph, we can take $$p=\frac{c}{8K^{rs+r+s}|V(H)|^{r-1}}\delta$$
	and
	$$q=\frac{c^2}{8K^{rs+r}}|X|\delta^{s+1}n^{-s}$$
	and we just need to verify that $q\geq C_0$ and $pq^s\geq C_0|X|^s$, where $C_0=C_0(H)$ is the constant provided by Lemma \ref{lem:dependent random choice}. 
	
	But, by equation (\ref{eqn:lower bound on |X|}),
	$$q\geq \frac{c^3}{16K^{rs+r+s+1}|V(H)|^{r-1}}\delta^{rs+r}n^{1-rs}\geq \frac{c^3}{16K^{rs+r+s+1}|V(H)|^{r-1}}C^{rs+r}.$$
	When $C$ is sufficiently large, this is indeed at least $C_0$.
	Moreover,
	\begin{align*}
		pq^s
		&= \frac{c^{2s+1}}{8^{s+1}K^{rs+r+s+s(rs+r)}|V(H)|^{r-1}}\delta^{s^2+s+1}n^{-s^2}|X|^s \\
		&\geq \frac{c^{2s+1}}{8^{s+1}K^{rs+r+s+s(rs+r)}|V(H)|^{r-1}}C^{s^2+s+1}n^{(s^2+s+1)(1-\frac{r+1}{rs+r})-s^2}|X|^s.
	\end{align*}
	Since $r\geq s+2$, we have $(s^2+s+1)(1-\frac{r+1}{rs+r})-s^2\geq 0$, so we get that
	$$pq^s\geq \frac{c^{2s+1}}{8^{s+1}K^{rs+r+s+s(rs+r)}|V(H)|^{r-1}}C^{s^2+s+1}|X|^s\geq C_0|X|^s,$$
	provided that $C$ is sufficiently large. Hence,  we can indeed apply Lemma \ref{lem:dependent random choice} to find a copy of $H$ in $G$, which is a contradiction.
\end{proof}

\section{Concluding remarks}

Let $T_{r,s, s'}$ be the rooted tree obtained from $F_{r,s}$ by attaching $s'$ leaves to the vertex~$y$, all of which are taken to be roots. It is easy to verify that $T_{r,s, s'}$ is balanced if and only if $s'-1\leq s\leq r+s'$. 
In their paper, Jiang, Jiang and Ma \cite{JJM20} actually studied this family of graphs, which clearly includes $F_{r,s} = T_{r,s,0}$, showing that if $T_{r,s, s'}$ is balanced and $r \geq s^3 - 1$, then $\ex(n,T_{r,s,s'}^t)=O(n^{2-1/\rho})$ holds, where $\rho=\rho(T_{r,s,s'}) = \frac{rs+r+s'}{r+1}$. We can prove the same upper bound under the relaxed condition $r\geq s-s'+1$ (except in the case $s'=0$, where we need $r\geq s+2$), almost matching the inequality $r\geq s-s'$ required for balancedness. 

\begin{theorem}
    For any integers $s'\geq 1$, $s\geq s'-1$, $r\geq s-s'+1$ and $t\geq 1$, $\ex(n,T_{r,s,s'}^t)=O(n^{2-\frac{r+1}{rs+r+s'}})$.
\end{theorem}

\noindent \emph{Proof sketch.} Since the proof is very similar to that of Theorem \ref{thm:main2},  we only mention the necessary adjustments. Taking $H=T_{r,s,s'}^t$, Lemma \ref{lem:heavy stars} still holds, although in the proof we need to consider the common neighbourhood of $s'$ vertices rather than that of a single vertex. The auxiliary hypergraphs $\mathcal{G}$ and $\mathcal{H}$ can then be defined identically (except that the vertex set of $\mathcal{G}$ is the common neighbourhood of $s'$ vertices). By making use of the extra $s'$ vertices whose common neighbourhood we considered, the existence of a subgraph $\mathcal{H}$ inside $\mathcal{G}$ still provides a copy of $H$.

The next substantial change is in Definition \ref{def:locus}, where an additional $s'$ vertices are taken as inputs, corresponding to the images of the $s'$ new leaves, and the vertices in $S$ are required to be common neighbours of these $s'$ vertices (on top of the previous requirements). Similarly, for the claim in (the analogue of) Theorem \ref{thm:regular}, we choose and fix the $s'$ new leaves as well as the $(r-1)s$ leaves that were fixed before.

Finally, although Lemma \ref{lem:dependent random choice} does not directly provide a copy of $H=T_{r,s,s'}^t$, we can still use Lemma~\ref{lem:dependent random choice} in the proof of Theorem~\ref{thm:regular} to find a copy of $F_{r,s}^t$ in $G[X,Y]$ with the $t$ copies of $y$ embedded into $X$. But $X$ is the common neighbourhood of $s'$ fixed vertices, so using those vertices we can extend $F_{r,s}^t$ to $H$.

The remaining changes are numerical, so we do not detail them here. \hfill $\Box$



\section*{Acknowledgments} 
We are grateful to the anonymous reviewers for several helpful comments.

\bibliographystyle{amsplain}


\begin{aicauthors}
\begin{authorinfo}[conlon]
  David Conlon\\
  Department of Mathematics\\
  California Institute of Technology\\
  Pasadena, CA 91125, USA\\
  dconlon\imageat{}caltech\imagedot{}edu
\end{authorinfo}
\begin{authorinfo}[janzer]
  Oliver Janzer\\
  Trinity College\\
  University of Cambridge\\
  Cambridge, United Kingdom\\
  oj224\imageat{}cam\imagedot{}ac\imagedot{}uk
\end{authorinfo}
\end{aicauthors}

\end{document}